\documentclass[12pt]{article}
\usepackage{amsfonts, amsmath, amssymb, amsthm}
\usepackage{color}
\usepackage{dsfont}
\usepackage[T1]{fontenc}
\usepackage{authblk}
\usepackage{graphicx,tikz,xcolor}
\usepackage{xcolor,eucal,enumerate,mathrsfs}
\usepackage[normalem]{ulem}
\usepackage{epsfig} 
\usepackage[latin1]{inputenc}
\usepackage{psfrag}          
\usepackage[ruled,vlined]{algorithm2e}
\usepackage{cancel}
\usepackage{setspace}

\usepackage{mystyle}
\newcommand{\bra}[1]{\left \langle #1 \right |}
\newcommand{\ket}[1]{\left | #1 \right \rangle}

\DeclareMathOperator{\Sp}{Sp}

\renewcommand{\d}{\dif}

\renewcommand{\d}{\dif}
\DeclareMathOperator{\ep}{\varepsilon}
\DeclareMathOperator{\Dept}{\cD_{\ep}^{\tau_1, \tau_2}}
\DeclareMathOperator{\Fept}{\cF_{\ep}^{\tau_1, \tau_2}}
\DeclareMathOperator{\Tr}{\mathrm{Tr}}

\newcommand{\cD}{\mathcal{D}}

\newcommand{\cX}{\mathcal{X}}

\newcommand{\cE}{\mathcal{E}}
\newcommand{\cH}{\mathcal{H}}

\newcommand{\cM}{\mathcal{M}}

\newcommand{\cS}{\mathcal{S}}
\newcommand{\cT}{\mathcal{T}}
\newcommand{\cF}{\mathcal{F}}

\newcommand{\cO}{\mathcal{O}}

\newcommand{\rmB}{\mathrm{B}}
\newcommand{\rmH}{\mathrm{H}}
\newcommand{\SDP}{\mathrm{H}_{\geq}}
\newcommand{\DP}{\mathrm{H}_>}

\usepackage{appendix}
\usepackage{enumitem}

\renewcommand{\d}{{\mathrm d}}

\newcommand{\suchthat}{\ensuremath{\ : \ }} 

\newcommand{\restr}[1]{\lower3pt\hbox{$|_{#1}$}}

\newcommand{\argmmax}{\mathop{\mathrm{argmax}}}

\theoremstyle{plain}
\newtheorem{theo}{Theorem}[section]
\newtheorem{proposition}{Proposition}[section]

\theoremstyle{definition}

\newtheorem{rem}[theo]{Remark}

\newcommand{\UQOT}{\mathcal{QOT}_{\varepsilon}^{\tau_1,\tau_2}}
\newcommand{\QOT}{\mathcal{QOT}_{\varepsilon}}

\DeclareFontFamily{U}{mathx}{\hyphenchar\font45}
\DeclareFontShape{U}{mathx}{m}{n}{
	<5> <6> <7> <8> <9> <10>
	<10.95> <12> <14.4> <17.28> <20.74> <24.88>
	mathx10
}{}
\DeclareSymbolFont{mathx}{U}{mathx}{m}{n}
\DeclareMathSymbol{\bigtimes}{1}{mathx}{"91}

\usepackage{scalerel}

\usepackage{color}
\usepackage{pstricks}
\usepackage{ifthen}
\newrgbcolor{kascolor}{.7 .1 .1}
\newrgbcolor{olicolor}{.1 .7 .1}
\newrgbcolor{augucolor}{.1 .1 .7}
\usepackage{csquotes}

\newboolean{DEBUG}
\setboolean{DEBUG}{true}

\ifthenelse {\boolean{DEBUG}}
{}

\ifthenelse {\boolean{DEBUG}}
{}

\ifthenelse {\boolean{DEBUG}}
{}

\ifthenelse {\boolean{DEBUG}}
{}

\title{Non-commutative Optimal Transport  \\ for semi-definite positive matrices}

\author[1,2,3]{Augusto Gerolin}
\author[1]{Nataliia Monina}
\affil[1]{Department of Mathematics and Statistics, University of Ottawa}
\affil[2]{Department of Chemistry and Biomolecular Sciences, University of Ottawa}
\affil[3]{Nexus for Quantum Technologies, University of Ottawa}

\begin{document}
\maketitle
\begin{abstract}
\noindent
We introduce the von Neumann entropy regularization of Unbalanced Non-commutative Optimal Transport, specifically Non-commutative Optimal Transport between semi-definite positive matrices (not necessarily with trace one). We prove the existence of a minimizer, compute the weak dual formulation and prove $\Gamma$-convergence results, demonstrating convergence to both Unbalanced Non-commutative Optimal Transport (as the Entropy-regularization parameter tends to zero) and von Neumann entropy regularized Non-commutative Optimal Transport problems (as the unbalanced penalty parameter tends to infinity). To draw an analogy to the Non-commutative case, we provide a concise introduction of the static formulation of Unbalanced Optimal Transport between positive measures and bounded cost functions.  
\end{abstract}

\section{Introduction}

Let $\ep>0$ be a positive number, $\cH_1$ and $\cH_2$ be finite-dimensional Hilbert spaces, $C\in \rmH(\cH_1\otimes\cH_2)$ be a Hermitian operator on $\cH_1\otimes\cH_2$, $\rho$ and $\sigma$ be density matrices, respectively, on $\cH_1$ and $\cH_2$. The von Neumann entropy regularized Non-commutative Optimal Transport \cite{FelGerPor23} is given by
\begin{equation}\label{QOT}
\QOT[\rho,\sigma] = \inf \left\lbrace \Tr[C\Gamma] + \ep\cS[\Gamma] \suchthat \Gamma \mapsto (\rho,\sigma) \right\rbrace,
\end{equation}
where $\cS[\Gamma]=\Tr[\Gamma\left(\log\Gamma-\Id\right)]$ is the von Neumann Entropy and $\Gamma\mapsto(\rho,\sigma)$ denotes the set of density matrices $\Gamma$ in $\cH_1\otimes\cH_2$ having partial traces given by $\rho$ and $\sigma$.

The static formulations of Non-commutative Optimal Transport (e.g., \cite{CalGolPau18,CalGolPau19,ColeEckFriZyc-MPAG-23, DPaTre19}) and its von Neumann entropy regularized counterpart \cite{FelGerPor23, Por23} are motivated by extending the (static) Optimal Transport Theory for probability measures  (e.g., \cite{AGS,CutPeyBook,Santabook,VilON}) into quantum states (e.g., density operators or density matrices).

In this note instead, we focus on the extension of the static formulation of von Neumann entropy regularized Non-commutative Optimal Transport for positive semi-definite and Hermitian operators (or Unbalanced Non-commutative Optimal Transport) given by
\begin{equation}\label{UQOT}
\UQOT[\rho,\sigma] = \inf \left\lbrace \Tr[C\Gamma] + \ep\cS[\Gamma] + \tau_1\cE[\Gamma_1|\rho] + \tau_2\cE[\Gamma_2|\sigma] \suchthat \Gamma\geq 0 \text{ and } \Gamma^*=\Gamma \right\rbrace,
\end{equation}
where $\tau_1,\tau_2>0$ are positive numbers, $C\in\rmH(\cH_1\otimes\cH_2)$ is a Hermitian operator in $\cH_1\otimes\cH_2$, $\rho\in\SDP(\cH_1)$ and $\sigma\in\SDP(\cH_2)$ are Hermitian semi-definite positive operators on $\cH_1$ and $\cH_2$ respectively,  $\Gamma_1,\Gamma_2$ are, respectively, the partial traces of $\Gamma_1 = \Tr_2[\Gamma]$, $\Gamma_2 = \Tr_1[\Gamma]$ in $\cH_2$ and $\cH_1$. Finally, the functional $\cE[\gamma_1|\gamma_2]$ is the Umegaki relative entropy between semi-definite positive and Hermitian operators $\gamma_1$ and $\gamma_2$
\[
\cE[\gamma_1|\gamma_2] = 
	\begin{cases}
		\Tr[\gamma_1(\log\gamma_1 - \log\gamma_2 - \Id)+\gamma_2] &\text{if } \ker \gamma_1 \subset \ker \gamma_2, \\
		+\infty &\text{otherwise}.
	\end{cases}
\]

The (von Neumann entropy regularized) Unbalanced Non-commutative Optimal Transport \eqref{UQOT} relaxes the constraint $\Gamma\mapsto(\rho,\sigma)$ in \eqref{QOT} and, in particular, the trace of the matrices $\rho$ and $\sigma$ must be finite (and not necessarily equal to one), allowing for unbalanced semi-definite positive and Hermitian operators.

The classical theory of Unbalanced Optimal Transport between positive measures has been introduced, independently, by Liero, Mielke and Savar\'e \cite{LiMiSa}, Chizat, Peyr{\'e}, Schmitzer and Vialard \cite{ChiPeyCom}, and Kondratyev, Monsaingeon and Vorotnikov \cite{KonMonVor16,KonMonVor16-arX}. This work proposes a generalization of the (static) Shannon-Entropy regularization of Unbalanced Optimal Transport introduced in \cite{FroZhaMobAraPog15} into the non-commutative setting. 

\noindent
\textbf{Main contributions:} The main results of this paper can be described as follows: (i) We provide an alternative proof of the existence (Theorem \ref{eq:dualunbalanced}) and characterization (Proposition \ref{prop:equiv_comp}) of the minimizer for the Shannon Entropy-regularized Unbalanced Optimal Transport for positive measures. (ii) We introduce the von Neumann entropy regularized Unbalanced Non-commutative Optimal Transport, prove the weak duality between the primal and dual problems (Theorem \ref{thm:UQOT_weakdual}), show the existence of the minimizer in \eqref{UQOT} (Proposition \ref{Fept:minimizer}), and, finally, we prove the $\Gamma-$convergence results (e.g., convergence of the minima) for \eqref{UQOT} (Theorem \ref{thm:UQOT_convergence_minimizer}) when the regularization parameters $\tau_1=\tau_2\to+\infty$ $(\ep>0$ fixed$)$ and when $\ep\to 0^+$ ($\tau_1,\tau_2>0$ are both fixed$)$. While our analysis focuses on finite-dimensional Hilbert spaces, some of our techniques are dimension-free, which are of notable significance in their own right. The exploration of the infinite-dimensional scenario remains a subject for future investigation.


\noindent
\textbf{Methodology and organization of the paper:} In section \ref{sec:unbalanced} we give a concise introduction of unbalanced optimal transport between positive measures and bounded costs. Our approach follows a similar duality strategy employed in \cite{DMaGer19,DMaGer20}. In section \ref{sec:unbalancedQOT}, we introduce the Unbalanced Non-commutative Optimal Transport problem, obtain weak duality results using the Legendre-Fenchel transforms, as well as prove the existence of the minimizer in \eqref{UQOT} by the direct method of Calculus of Variations. Finally, in section \ref{sec:unbalancedgammaQOT} we directly prove the $\Gamma-$convergence of \eqref{UQOT} to both von Neumann entropy regularized Non-commutative Optimal transport (as $\tau_1=\tau_2\to+\infty$) and Unbalanced Non-commutative Optimal Transport (as $\ep\to 0$), and by showing that the primal functional in \eqref{UQOT} is equi-coercive with respect to regularization parameters, we obtain the convergence of $\argmin$ of \eqref{UQOT}. 

\section{Unbalanced Optimal Transport}\label{sec:unbalanced}

This section aims to provide a brief, self-contained introduction of (static) Unbalanced Optimal Transport between positive measures, to draw an analogy to the Non-commutative case in section \ref{sec:unbalancedQOT}. A similar approach has also been considered in \cite{BiervonNeuSte-JMIV-23}, including proof of convergence for the Unbalanced Sinkhorn algorithm and the Unbalanced multi-marginal optimal transport theory.


Let $X$ and $Y$ be complete separable metric spaces, $\mu\in\mathcal{M}_+(X)$ and $\nu\in\mathcal{M}_+(Y)$ be positive measures, and let $c:X\times  Y \to\R$ be a measurable function on $X\times Y$. The Unbalanced Optimal Transport (UOT) problem is defined by
\begin{equation}\label{eq:UOTclassical}
\cO\cT^{\tau_1,\tau_2}(\mu,\nu) = \inf_{\gamma\in\mathcal{M}_{+}( X \times Y )} \int\limits_{ X \times Y }c(x,y)d\gamma + \tau_1\KL((e_1)_{\sharp}\gamma|\mu) + \tau_2\KL((e_2)_{\sharp}\gamma|\nu),
\end{equation}
where $e_1:X\times Y\to X$ and $e_2:X\times Y\to Y$ are the projection operators, i.e. $e_1(x,y) = x$, $e_2(x,y)=y$, and  $(e_i)_{\sharp}\gamma,$ for $ i=1,2$ denotes the push-forward of the measure $\gamma$ via the projector operator $e_i$. The functional $\KL$ is the Kullback-Leibler divergence between two positive measures $\alpha$ and $\beta$
\begin{equation}\label{eq:classicalKL}
\KL(\alpha|\beta) =\begin{cases} \ds\int \left[\dfrac{d\alpha}{d\beta}\left(\log\left(\dfrac{d\alpha}{d\beta}\right)-1\right)+1\right] d\beta, & \mbox{if $\alpha \ll \beta$}, \\
            +\infty & \mbox{otherwise},\end{cases}
\end{equation}
where $\frac{d\alpha}{d\beta}$ denotes the Radon-Nikodym derivative of $\alpha$ with respect to $\beta$.

The main idea behind UOT is to introduce a so-called \textit{slack variable} that represents the discrepancy between the masses of the source and target distributions. The UOT functional \eqref{eq:UOTclassical} extends classical Optimal Transport theory to handle scenarios where the mass (or density) of the data $\mu,\nu$ are positive measures rather than probability ones. 

Therefore, the problem \eqref{eq:UOTclassical} relaxes the constraint that the total mass of the source and target distributions must be equal, allowing for imbalanced mass distributions, in the sense that if $\mu,\nu$ are probability distributions, the UOT problem \eqref{eq:UOTclassical} reduces to the optimal transport problem when $\tau_1,\tau_2\to +\infty$

\begin{equation}\label{eq:OTclassical}
\cO\cT(\mu,\nu) = \inf_{\gamma\in\mathcal{M}_{+}( X \times Y )} \int\limits_{ X \times Y }c(x,y)d\gamma + i((e_1)_{\sharp}\gamma|\mu) + i((e_2)_{\sharp}\gamma|\nu),
\end{equation}
where $i$ denotes the indicator function, i.e. $i(\alpha|\beta) = 0$ if $\alpha = \beta$ and $i(\alpha|\beta)=+\infty$ otherwise.


\subsection{Shannon-Entropy regularized Unbalanced Optimal Transport}

Let $\ep, \tau_1, \tau_2>0$ be positive numbers, $X$ and $Y$ be complete separable metric spaces. The Shannon entropy regularization of the Unbalanced Optimal Transport between positive measures $\mu\in\mathcal{M}_{+}( X )$ and $\nu\in\mathcal{M}_{+}( Y )$ is given by
\begin{equation}\label{eq:UOT}
\cO\cT^{\tau_1,\tau_2}_{\ep}(\mu,\nu) = \inf_{\gamma\in\mathcal{M}_+( X \times Y)} \int\limits_{ X \times Y}cd\gamma + \ep S(\gamma) + \tau_1\KL((e_1)_{\sharp}\gamma|\mu) + \tau_2\KL((e_2)_{\sharp}\gamma|\nu),
\end{equation}
where $c: X \times Y \to\R$ is a measurable cost function and $\KL$ is the Kullback-Leibler divergence \eqref{eq:classicalKL} between two positive measures $\alpha$ and $\beta$.  The functional $S:\mathcal{M}_{+}( X \times Y)\to\R$ is the Shannon-entropy 
\begin{equation}\label{eq:shannon_entropy}
    S(\gamma) = \begin{cases} \ds\int_{X\times Y} \frac{d\gamma}{d(\mu\otimes\nu)}\left(\log\left(\frac{d\gamma}{d(\mu\otimes\nu)}\right)-1\right)\mu\otimes\nu & \mbox{if $\gamma \ll \mu\otimes\nu$} \\
            +\infty & \mbox{otherwise} \end{cases},
\end{equation}
where $\frac{d\gamma}{d(\mu\otimes\nu)}$ denotes the Radon-Nikodym derivative of $\gamma$ with respect to $\mu\otimes\nu$.

The existence of a minimizer in \eqref{eq:UOT} is guaranteed, for instance, when the cost $c$ is nonnegative and lower semicontinuous \cite{ChiPeyCom}. 

\subsection{Dual problem}
Let $X,~Y$ be complete separable metric spaces, $\mu\in \cM_+( X ), \nu\in\cM_+( Y )$ be positive measures and $c\in L^{\infty}(X\times Y)$ be a cost function. For given $\ep, \tau_1, \tau_2>0$ positive numbers, we define the unbalanced dual functional $D^{\tau_1, \tau_2}_{\ep}:L^\infty( X )\times L^\infty( Y )\to \R$ as
\begin{equation*}
    D^{\tau_1, \tau_2}_{\ep}(u, v) = -\tau_1\int\limits_{ X } (e^{-\frac{u}{\tau_1}}-1) \d\mu -\tau_2\int\limits_{ Y } (e^{-\frac{v}{\tau_2}}-1)\d\nu - 
    \ep \int\limits_{ X \times Y } e^{\frac{u+v-c}{\ep}}\d\mu\otimes\nu.
\end{equation*}
The corresponding dual (Kantorovich) problem takes the form
\begin{equation}\label{eq:dualunbalanced}
\sup\lbrace D^{\tau_1, \tau_2}_{\ep}(u,v)\suchthat u\in L^\infty( X ), v\in L^\infty( Y )\rbrace. 
\end{equation}



In the following, we will prove the existence of a maximizer in \eqref{eq:dualunbalanced} via the direct method of Calculus of Variations. We follow the approach introduced in \cite{DMaGer19} and define the Unbalanced Entropic $c-$transform.

\begin{deff}[Unbalanced Entropic $c$-transform or $(c,\tau,\ep)$-transform]
Let $\ep,\tau_1, \tau_2>0$ be the positive parameters, $X$ and $Y$ be complete separable metric spaces, $\mu \in \cM_+( X ),~ \nu \in \cM_+( Y )$ be positive measures, and let $c\in L^\infty(X\times Y)$ be a cost function. Given $u\in L^\infty( X )$ and $v\in  L^\infty( Y )$, the entropic $(c,\tau_2,\ep)$-transform of a function $u$ is defined by a map $(\cdot)^{(c,\tau_2,\ep)}:L^\infty( X )\to L^{\infty}( Y )$ as

\begin{equation}
    \label{u:transform}
    u^{(c,\tau_2,\ep)}(y) =
    -\dfrac{\tau_2\ep}{\tau_2+\ep}
        \log \int\limits_{ X }
            \exp 
            \left(
                \frac{u(x) - c(x,y)}{\ep}
            \right) \d\mu(x),
\end{equation}

and analogously for $v$, we define the map $(\cdot)^{(c,\tau_1,\ep)}:L^\infty( Y )\to L^{\infty}( X )$ as

\begin{equation}
    \label{v:transform}
    v^{(c,\tau_1,\ep)}(x) =
    -\dfrac{\tau_1\ep}{\tau_1+\ep} 
        \log \int\limits_{ Y }
            \exp 
            \left(
                \frac{v(y) - c(x,y)}{\ep}
            \right) \d\nu(y).
\end{equation}

\end{deff}
Note that when $\tau_1,\tau_2\to+\infty$, the $(c,\tau,\ep)$-transforms \eqref{u:transform} and \eqref{v:transform} become the classical $(c,\ep)-$transform (or Sinkhorn iterations) \cite{DMaGer19} for the Shannon-Entropy regularized optimal transport problem \eqref{eq:OTclassical}.


The following proposition shows that the Unbalanced $(c,\tau,\ep)-$transforms are well-defined.


\begin{proposition}\label{prop:aprioribdd}
Let $\ep,\tau_1, \tau_2>0$ be the positive parameters, $X$ and $Y$ be complete separable metric spaces, $\mu \in \cM_+( X ),~ \nu \in \cM_+( Y )$ be positive measures, and let $c\in L^\infty(X\times Y)$ be a cost function. If $u \in L^\infty( X ),~ v\in  L^\infty( Y )$, then $u^{(c,\tau_2,\ep)}\in L^{\infty}( Y )$ and $v^{(c,\tau_1,\ep)}\in L^{\infty}( X )$.
\end{proposition}

\begin{proof}
It suffices to show that the integral part of the transform is bounded, exploiting the fact that $||c||_\infty, ||u||_\infty, ||v||_\infty<\infty$. The following estimates will be shown for $u^{(c,\tau_2,\ep)}$, and the identical bound will also hold for $v^{(c,\tau_1,\ep)}$. 

By monotonicity of the exponential, the inequality
\[
\exp 
            \left(
                \frac{-||u||_\infty - ||c||_\infty}{\ep} 
            \right) \leq
    \exp 
\left(
    \frac{u(x) - c(x,y)}{\ep} 
\right) \leq 
\exp 
            \left(
                \frac{||u||_\infty+ ||c||_\infty}{\ep} 
            \right)
\]
holds $\mu$ and $\nu-$almost everywhere, and since $\mu(X)$ is a finite positive value,  taking the integrals and logarithms on both sides we easily obtain 
\[
    \left|\left|u^{(c,\tau_2,\ep)} +\dfrac{\tau_2\ep}{\tau_2+\ep}\log(\mu(X))\right|\right|_\infty
    \leq \dfrac{\tau_2\ep}{\tau_2+\ep}\left(\frac{||u||_\infty + ||c||_\infty}{\ep}\right),
\]

which concludes the proof.
\end{proof}

The next lemma shows that the Unbalanced  Entropic $c$-transform increases the value of the dual functional \eqref{eq:dualfunctional}. 

\begin{lemma}\label{lemma:increasedual}
Let $\ep,\tau_1, \tau_2>0$ be the positive parameters, $X$ and $Y$ be complete separable metric spaces, $\mu \in \cM_+( X ),~ \nu \in \cM_+( Y )$ be positive measures, and let $c\in L^\infty(X\times Y)$ be a cost function. Assume that $u \in L^\infty( X ), v \in L^\infty( Y )$. Then:
\begin{equation}\label{eq:Dutransform}
    D^{\tau_1,\tau_2}_{\ep}(u, v) \leq D^{\tau_1,\tau_2}_{\ep}(u, u^{(c,\tau_2,\ep)}), \text{ for all } u\in L^\infty( X ),
\end{equation}
\begin{equation}\label{eq:optcond}
    D^{\tau_1,\tau_2}_{\ep}(u, v) = D^{\tau_1,\tau_2}_{\ep}(u, u^{(c,\tau_2,\ep)}), \text{ iff } v = u^{(c,\tau_2,\ep)}.
\end{equation}

and, analogously,
\begin{equation}\label{eq:Dvtransform}
    D^{\tau_1,\tau_2}_{\ep}(u, v) \leq D^{\tau_1,\tau_2}_{\ep}(v^{(c,\tau_1,\ep)}, v), \text{ for all } v\in L^\infty( Y ),
\end{equation}
\begin{equation}\label{eq:optcond2}
    D^{\tau_1,\tau_2}_{\ep}(u, v) = D^{\tau_1,\tau_2}_{\ep}(v^{(c,\tau_1,\ep)}, v), \text{ iff } u = v^{(c,\tau_1,\ep)}.
\end{equation}

\end{lemma}
\begin{proof}
Consider the dual functional
\begin{equation*}
\begin{aligned}
    D^{\tau_1, \tau_2}_{\ep}(u, v) =& -\tau_1\int\limits_{ X } (e^{-\frac{u}{\tau_1}}-1) \d\mu -\tau_2\int\limits_{ Y } (e^{-\frac{v}{\tau_2}}-1)\d\nu - 
    \ep \int\limits_{ X \times Y } e^{\frac{u+v-c}{\ep}}\d\mu\otimes\nu= \\
    =& -\tau_1\int\limits_{ X } (e^{-\frac{u}{\tau_1}}-1) \d\mu -\tau_2\int\limits_{ Y } (e^{-\frac{v}{\tau_2}}-1)\d\nu - 
    \ep \int\limits_{Y }e^{\frac{v}{\ep}} \left[\int_X e^{\frac{u-c}{\ep}}\d\mu\right] \d\nu,
\end{aligned}
\end{equation*}
and using the entropic transform (\ref{u:transform}) we can express the integral 
\begin{equation*}
\int_X e^{\frac{u-c}{\ep}}\d\mu = \exp\left(-\frac{(\tau_2+\ep)u^{(c,\tau_2,\ep)}}{\tau_2\ep}\right),
\end{equation*}
thus
\begin{equation*}
\begin{aligned}
    D^{\tau_1, \tau_2}_{\ep}(u, v) =& -\tau_1\int\limits_{ X } (e^{-\frac{u}{\tau_1}}-1) \d\mu -\tau_2\int\limits_{ Y } (e^{-\frac{v}{\tau_2}}-1)\d\nu - 
    \ep \int\limits_{Y }e^{\frac{v}{\ep}} e^{-\frac{(\tau_2+\ep)u^{(c,\tau_2,\ep)}}{\tau_2\ep}} \d\nu \\
    =&-\tau_1\int\limits_{ X } (e^{-\frac{u}{\tau_1}}-1) \d\mu -\int\limits_{ Y } \left[\tau_2(e^{-\frac{v}{\tau_2}}-1) + 
    \ep e^{\frac{v}{\ep}-\frac{(\tau_2+\ep)u^{(c,\tau_2,\ep)}}{\tau_2\ep}} \right]\d\nu.
\end{aligned}
\end{equation*}
Invoking the variational principle one can show that, in fact, 
\begin{equation*}
    u^{(c,\tau_2,\ep)}\in \argmin\limits_{v\in L^\infty( Y )} \int\limits_{ Y } \left[\tau_2(e^{-\frac{v}{\tau_2}}-1) + 
    \ep e^{\frac{v}{\ep}-\frac{(\tau_2+\ep)u^{(c,\tau_2,\ep)}}{\tau_2\ep}} \right]\d\nu,
\end{equation*}
and in particular, since the function $h(t) = \tau_2e^{-\frac{t}{\tau_2}} + 
    \ep e^{\frac{t}{\ep}-\frac{(\tau_2+\ep)a}{\tau_2\ep}}$ is strictly convex in $t$, the minimizer $t=a$ is unique. Therefore, \eqref{eq:Dutransform} and \eqref{eq:optcond} hold, and by an identical computation \eqref{eq:Dvtransform}, \eqref{eq:optcond2} hold as well.
\end{proof}

\begin{rem}
As a consequence of Lemma \ref{lemma:increasedual}, at least when the cost $c\in L^\infty(X\times Y)$, the Unbalanced Entropic $c-$transform can be equivalently defined as  
\begin{equation}\label{eq:ctauepabstract}
    v^{(c,\tau_1,\ep)} \in \argmmax\limits_{u\in L^\infty( X )}\lbrace D^{\tau_1,\tau_2}_{\ep}(u,v)\rbrace, ~~~~\text{and}~~~~~ u^{(c,\tau_2,\ep)} \in \argmmax_{v\in L^\infty( Y )}\lbrace D^{\tau_1,\tau_2}_{\ep}(u,v)\rbrace.
\end{equation}
\end{rem}

The following proposition will play the main role in showing the existence of the maximizer of the
dual problem.

\begin{proposition}\label{prop:sublevel_set_transform}Let $\ep,\tau_1, \tau_2>0$ be the positive parameters, $X$ and $Y$ be complete separable metric spaces, $\mu \in \cM_+( X ),~ \nu \in \cM_+( Y )$ be positive measures, and let $c\in L^\infty(X\times Y)$ be a cost function. Then for fixed $M\geq 0$ there exist constants $\alpha, \beta\in \R$ such that for any $u \in L^\infty( X ), v \in L^\infty( Y )$ satisfying 
\begin{equation}\label{eq:Dept_sublevel}
    D^{\tau_1,\tau_2}_{\ep}(u,v) \geq -M,
\end{equation}
holds that $||u^{(c,\tau_2,\ep)}||_{\infty}\leq \alpha$, and $||v^{(c,\tau_1,\ep)}||_{\infty}\leq \beta.$
\end{proposition}

\begin{proof}
Suppose $u$ and $v$ satisfy condition \eqref{eq:Dept_sublevel}, then we obtain the following relation
\[
0\leq \tau_1\int\limits_{ X } e^{-\frac{u}{\tau_1}} \d\mu +\tau_2\int\limits_{ Y } e^{-\frac{v}{\tau_2}}\d\nu + 
    \ep \int\limits_{ X \times Y } e^{\frac{u+v-c}{\ep}}\d\mu\otimes\nu \leq M + \tau_1\mu(X) + \tau_2\nu(Y) \equiv A.
\]
See that each of the integral terms is nonnegative, and thus bounded by $A$, as their sum is bounded from above by $A$. Consequently, we can consider these three inequalities individually.

In the following, we will exploit the convexity of the exponential function and apply Jensen's inequality to the normalized probability measures 
\[
\d\rho\equiv \dfrac{1}{\mu(X)}\d\mu \quad\quad\text{and} \quad\quad \d\sigma \equiv \dfrac{1}{\nu(Y)}\d\nu.
\]

For the first term, by the direct application of Jensen's inequality
\begin{align*}
    A\geq \tau_1\int\limits_{ X } e^{-\frac{u}{\tau_1}} \d\mu = \tau_1\mu(X)\int\limits_{ X } e^{-\frac{u}{\tau_1}} \d\rho \geq \tau_1\mu(X)\exp\left(-\frac{\int_X u\d\rho}{\tau_1}\right)>0,
\end{align*}
and, in particular, we see that 
\begin{align}\label{eq:L1_u_bound}
    \int\limits_{ X }u\d\rho \geq \tau_1 \log\left(\dfrac{\tau_1\mu(X)}{A}\right) \equiv \omega_1.
\end{align}

By an analogous computation, we also obtain
\begin{align}\label{eq:L1_v_bound}
    \int\limits_{ Y }v\d\sigma \geq \tau_2 \log\left(\dfrac{\tau_2\nu(Y)}{A}\right) \equiv \omega_2.
\end{align}

Now, write the remaining term as 
\begin{align*}
    A\geq \ep \int\limits_{ X \times Y } e^{\frac{u+v-c}{\ep}}\d\mu\otimes\nu = \ep \int\limits_{ X}e^{\frac{u}{\ep}} \left[\int\limits_Y e^{\frac{v-c}{\ep}}\d\nu\right] \d\mu  = \ep \int\limits_{ X}e^{\frac{u}{\ep}}\nu(Y) \left[\int\limits_Y e^{\frac{v-c}{\ep}}\d\sigma\right] \d\mu.
\end{align*}

Then by the monotonicity and convexity of the exponential, see that
\begin{align*}
    A\geq \ep \nu(Y)\int\limits_{ X}e^{\frac{u}{\ep}} \left[\int\limits_Y e^{\frac{v-||c||_\infty}{\ep}}\d\sigma\right] \d\mu \geq \ep\nu(Y) \int\limits_{ X}e^{\frac{u}{\ep}} \exp\left(\frac{\int_Y v\d\sigma -||c||_\infty}{\ep}\right)   \d\mu,
\end{align*}
and applying \eqref{eq:L1_v_bound} to the latter expression, by monotonicity we obtain
\begin{align*}
    A\geq \ep\nu(Y) \int\limits_{ X}e^{\frac{u}{\ep}} \exp\left(\frac{\omega_2 -||c||_\infty}{\ep}\right)   \d\mu,
\end{align*}
or, equivalently,
\begin{align}\label{eq:Lexp_u_bound}
    \int_X e^{\frac{u}{\ep}}\d\mu \leq \dfrac{A}{\ep \nu(Y)} e^{-\frac{\omega_2-||c||_\infty}{\ep}}.
\end{align}

By symmetry, we also get similar bounds for $v$
\begin{align}\label{eq:Lexp_v_bound}
    \int_Y e^{\frac{v}{\ep}}\d\nu \leq \dfrac{A}{\ep \mu(X)} e^{-\frac{\omega_1-||c||_\infty}{\ep}}.
\end{align}

Finally, we can now find the upper and lower bounds for the integral parts of $u^{(c,\tau_2,\ep)}$ and $v^{(c,\tau_1,\ep)}$ that are independent of $u$ and $v$. 

Using \eqref{eq:Lexp_u_bound}, see that 
\begin{align*}
    \int\limits_{ X }e^{
    \frac{u - c}{\ep}} \d\mu \leq \int\limits_{ X }e^{
    \frac{u + ||c||_{\infty}}{\ep}} \d\mu \leq \dfrac{A}{\ep \nu(Y)} \exp\left(-\frac{\omega_2-2||c||_\infty}{\ep}\right).
\end{align*}
Conversely, applying Jensen's inequality once more, along with \eqref{eq:L1_u_bound}
\begin{equation*}
\begin{aligned}
    \int\limits_{ X }e^{
    \frac{u - c}{\ep}} \d\mu\geq \mu(X)\int\limits_{ X }e^{
    \frac{u - ||c||_{\infty}}{\ep}} \d\rho \geq \mu(X) \exp\left(\frac{\int_X u \d\rho - ||c||_{\infty}}{\ep}\right) \geq \mu(X) \exp\left(\frac{\omega_1 - ||c||_{\infty}}{\ep}\right). 
\end{aligned}
\end{equation*}

By identical computations for $v$, we can finally estimate the bounds $\mu$ and $\nu-$a.e.
\[
   -\dfrac{\tau_2\ep}{\tau_2+\ep} \left[ \log\left(\dfrac{A}{\ep\nu(Y)}\right) - \dfrac{\omega_2-2||c||_\infty}{\ep}\right]\leq
   u^{(c,\tau_2,\ep)} \leq 
   -\dfrac{\tau_2\ep}{\tau_2+\ep} \left[ \log(\mu(X)) + \dfrac{\omega_1-||c||_\infty}{\ep}\right],
\]
\[
   -\dfrac{\tau_1\ep}{\tau_1+\ep} \left[ \log\left(\dfrac{A}{\ep\mu(X)}\right) - \dfrac{\omega_1-2||c||_\infty}{\ep}\right]\leq
   v^{(c,\tau_1,\ep)} \leq 
   -\dfrac{\tau_1\ep}{\tau_1+\ep} \left[ \log(\nu(Y)) + \dfrac{\omega_2-||c||_\infty}{\ep}\right],
\]
which concludes the proof, as the boundary values are independent of the choice of $u$ and $v$ that satisfy \eqref{eq:Dept_sublevel}.
\end{proof}

We are now in position to show the existence of the maximizer in the dual problem \eqref{eq:dualunbalanced}. 

\begin{teo}
Let $\ep,\tau_1, \tau_2>0$ be the positive parameters, $X$ and $Y$ be complete separable metric spaces, $\mu \in \cM_+( X ),~ \nu \in \cM_+( Y )$ be positive measures, and let $c\in L^\infty(X\times Y)$ be a cost function. Then the dual problem \eqref{eq:dualunbalanced}
$$\sup\lbrace D^{\tau_1, \tau_2}_{\ep}(u,v)\suchthat u\in L^\infty( X ), \, v\in L^\infty( Y )\rbrace.$$
admits a maximizer.
\end{teo}

\begin{proof}
First, notice that $D^{\tau_1,\tau_2}_{\ep}(u,v) \leq \tau_1 \mu(X) + \tau_2 \nu(Y)<\infty$. 
Now let $(u_n)_{n\in\N} \subset L^\infty( X ) $ and $(v_n)_{n\in \N} \subset L^\infty( Y )$ be  maximizing sequences for $D^{\tau_1,\tau_2}_{\ep}$.  By Lemma \ref{lemma:increasedual}, we have that $(v_n^{(c,\tau_1,\ep)})_{n\in\N} \subset L^\infty( X )$ and $(u_n^{(c,\tau_2,\ep)})_{n\in\N} \subset L^\infty( Y )$ are also maximizing sequences for $D^{\tau_1,\tau_2}_{\ep}$.

On the other hand, without loss of generality, we can assume that there exists a constant $M>0$ such that $D^{\tau_1,\tau_2}_{\ep}(u_n,v_n)>-M$ for any $n\in \N$. Then, due to Proposition \ref{prop:sublevel_set_transform} the sequences $(u_n^{(c,\tau_2,\ep)})_{n\in\N} \subset L^\infty( Y )$ and $(v_n^{(c,\tau_1,\ep)})_{n\in\N} \subset L^\infty( X )$ are bounded, and by Banach-Alaoglu theorem we can extract a common subsequence $(n_k)_{k\in\N}$ such that $v_n^{(c,\tau_1,\ep)}\rightharpoonup^* \bar{u}$ and $u_n^{(c,\tau_2,\ep)}\rightharpoonup^* \bar{v}$ for some $\bar{u}\in L^\infty(X)$ and $\bar{v}\in L^\infty(Y)$. 

Next, notice that since the maps $t \mapsto e^t$ and $t\mapsto e^{-t}-1$ are continuous and convex functions, by Fatou's Lemma we have
\begin{equation*}
    \begin{cases}
        \liminf\limits_{k\to\infty} \int\limits_{ X \times Y }e^{\frac{v^{(c,\tau_1,\ep)}_{n_k}+u^{(c,\tau_2,\ep)}_{n_k}-c}{\ep}}\d\mu\otimes\nu \geq  \int\limits_{ X \times Y }e^{\frac{\overline{u}+\overline{v}-c}{\ep}}\d\mu\otimes\nu,\\
        \liminf\limits_{k\to\infty}\int\limits_{ X } \left(e^{-\frac{v^{(c,\tau_1,\ep)}_{n_k}}{\tau_1}}-1\right) \d\mu \geq \int\limits_{ X } \left(e^{-\frac{\bar{u}}{\tau_1}}-1\right) \d\mu, \\
        \liminf\limits_{k\to\infty}\int\limits_{ Y } \left(e^{-\frac{u^{(c,\tau_2,\ep)}_{n_k}}{\tau_2}}-1\right) \d\nu \geq \int\limits_{ Y } \left(e^{-\frac{\bar{v}}{\tau_2}}-1\right) \d\nu,
    \end{cases}
\end{equation*}
and as the sequences of transforms also maximize the dual functional, we obtain the chain of inequalities
\begin{align*}
    \sup_{u,v} D^{\tau_1,\tau_2}_{\ep}(u,v) = \lim_{n\to\infty}D^{\tau_1,\tau_2}_{\ep}(v^{(c,\tau_1,\ep)}_n,u^{(c,\tau_2,\ep)}_n) = \limsup\limits_{k\to\infty}D^{\tau_1,\tau_2}_{\ep}(v^{(c,\tau_1,\ep)}_{n_k},u^{(c,\tau_2,\ep)}_{n_k}) \leq D^{\tau_1,\tau_2}_{\ep}(\overline{u},\overline{v}).
\end{align*}


Thus, $(\overline{u},\overline{v})$ is a maximizer for $D^{\tau_1,\tau_2}_{\ep}$, since both $\overline{u} \in L^\infty( X )$ and $ \overline{v} \in L^\infty( Y )$ are bounded by construction. Finally, the strict concavity of $D^{\tau_1,\tau_2}_{\ep}$ and Lemma \ref{lemma:increasedual} imply that the maximizer is unique and, in particular, $\overline{v} = \overline{u}^{(c,\tau_2, \ep)}$ and $\overline{u} = \overline{v}^{(c,\tau_1, \ep)}$.
\end{proof}

In the remaining part of this section, we will focus on establishing the duality between the primal \eqref{eq:UOT} and dual problems \eqref{eq:dualunbalanced} as well as on showing the existence of the minimizer in \eqref{eq:UOT}, following a similar strategy employed in \cite{DMaGer19}.

\begin{proposition}\label{prop:primaldual}
Let $\ep,\tau_1, \tau_2>0$ be the positive parameters, $X$ and $Y$ be complete separable metric spaces, $\mu \in \cM_+( X ),~ \nu \in \cM_+( Y )$ be positive measures, and let $c\in L^\infty(X\times Y)$ be a cost function. Then for every $\gamma\in\cM_+( X \times Y )$, $u\in L^\infty( X )$, $v\in L^\infty( Y )$ holds 
\begin{equation}\label{eq:primaldual}
\int\limits_{ X \times Y }c(x,y)d\gamma +\ep S(\gamma)+ \tau_1\KL((e_1)_{\sharp}\gamma|\mu) + \tau_2\KL((e_2)_{\sharp}\gamma|\nu) \geq D^{\tau_1,\tau_2}_{\ep}(u,v).
\end{equation}
\end{proposition}

\begin{proof}
By definition of the Shannon entropy in \eqref{eq:shannon_entropy}, we can assume that $\gamma$ is a positive measure and absolutely continuous with respect to $\mu\otimes \nu$. We denote by $p>0$ its density. Then, notice that $(e_1)_\sharp \gamma$ and $(e_2)_\sharp \gamma$ have densities $p_X$ and $p_Y$ with respect to $\mu$ and $\nu$, where
\[
p_X = \int_Y p \d\nu, \quad\quad \text{and} \quad\quad  p_Y = \int_X p \d\mu.
\]

Next, for the exponential terms we can use the Fenchel-Young inequality $e^t+s(\log s-1)\geq ts$, and for all $u\in L^\infty( X )$ and $v\in L^\infty( Y )$
\begin{align*}
D^{\tau_1,\tau_2}_{\ep}(u,v) =-&\tau_1\int\limits_{ X } (e^{-\frac{u}{\tau_1}}-1) \d\mu -\tau_2\int\limits_{ Y } (e^{-\frac{v}{\tau_2}}-1)\d\nu - 
    \ep \int\limits_{ X \times Y } e^{\frac{u+v-c}{\ep}}\d\mu\otimes\nu\leq \\
\leq &\tau_1\int\limits_X \left(\frac{u}{\tau_1}p_X +p_X(\log(p_X)-1)  +1\right)\d\mu  +\\
+&\tau_2\int\limits_Y \left(\frac{v}{\tau_2}p_Y +p_Y(\log(p_Y)-1)  +1\right)\d\nu + \\
+&\ep\int\limits_{X\times Y}\left(p(\log(p)-1) - \frac{u+v-c}{\ep}p\right)\d\mu\otimes\nu \leq\\
\leq &\tau_1\KL((e_1)_{\sharp}\gamma|\mu) + \tau_2\KL((e_2)_{\sharp}\gamma|\nu)  +\ep S(\gamma) +\ep\int_{ X \times Y }c\mathrm{d}\gamma,
\end{align*}
which completes the proof.
\end{proof}

\begin{proposition}[Characterization of \eqref{eq:UOTclassical}]\label{prop:equiv_comp}
Let $\ep,\tau_1, \tau_2>0$ be the positive parameters, $X$ and $Y$ be complete separable metric spaces, $\mu \in \cM_+( X ),~ \nu \in \cM_+( Y )$ be positive measures, and let $c\in L^\infty(X\times Y)$ be a cost function. Then the following are equivalent:
\begin{enumerate}
\item[(a)] \emph{(Maximizers)} $u^*$ and $v^*$ are maximizing potentials for \eqref{eq:dualunbalanced};
\item[(b)] \emph{(Maximality condition)} $(u^*)^{(c,\tau_2,\ep)}=v^*$ and $(v^*)^{(c,\tau_1,\ep)}=u^*$;
\item[(c)] \emph{(Duality attainment) } $\cO\cT^{\tau_1,\tau_2}_{\ep}(\mu,\nu) = D^{\tau_1,\tau_2}_{\ep}(u^*,v^*).$
\end{enumerate}
Moreover, $\gamma^*$ defined as $\d\gamma^*=\exp\left(\frac{u^*(x)+v^*(y)-c(x,y)}{\ep}\right)\d\mu\otimes\nu$ is the (unique) minimizer for the problem \eqref{eq:UOTclassical}.
\end{proposition}

\begin{proof}
Assume that $u^*$ and $v^*$ are maximizers of \eqref{eq:dualunbalanced}. We are going to prove that $v^*=(u^*)^{(c,\tau_2,\ep)}$. Due to Lemma \ref{lemma:increasedual}, we have $D^{\tau_1,\tau_2}_{\ep} (u^*, (u^*)^{(c,\tau_2,\ep)}) \geq D^{\tau_1,\tau_2}_{\ep}(u^*,v^*)$;  however, by the maximality of $u^*,v^*$ we have also $D^{\tau_1,\tau_2}_{\ep}(u^*,v^*) \geq D^{\tau_1,\tau_2}_{\ep} (u^*, (u^*)^{(c,\tau_2,\ep)}) $. These imply that $D^{\tau_1,\tau_2}_{\ep}(u^*,(u^*)^{(c,\tau_2,\ep)})=D^{\tau_1,\tau_2}_{\ep}(u^*,v^*)$. Therefore, by \eqref{eq:optcond} , $v^*=(u^*)^{(c,\tau_2,\ep)}$.  A similar argument also shows that $u^*=(v^*)^{(c,\tau_1,\ep)}$.  

Now, we assume that $u^*$ and $v^*$ are functions such that $v^*=(u^*)^{(c,\tau_2,\ep)}$ and $u^*=(v^*)^{(c,\tau_1,\ep)}$. We will prove that $\cO\cT^{\tau_1,\tau_2}_{\ep}(\mu,\nu) = D^{\tau_1,\tau_2}_{\ep}(u^*,v^*)$. For simplicity, we denote $F^{\tau_1,\tau_2}_{\ep}(\gamma)$ by
\[
F^{\tau_1,\tau_2}_{\ep}(\gamma) = \int\limits_{ X \times Y }c \d\gamma +\ep S(\gamma)+ \tau_1\KL((e_1)_{\sharp}\gamma|\mu) + \tau_2\KL((e_2)_{\sharp}\gamma|\nu).
\]
Let us define $\d\gamma^* = \exp(\frac{1}{\ep}\left(u^*+v^*-c\right))\d\mu\otimes\nu$. Due to Proposition \ref{prop:primaldual}, we have that $F^{\tau_1,\tau_2}_{\ep}(\gamma^*) \geq D^{\tau_1,\tau_2}_{\ep}(u,v)$, $\forall u \in L^\infty( X ), v \in L^\infty( Y )$ and $F^{\tau_1,\tau_2}_{\ep}(\gamma) \geq D^{\tau_1,\tau_2}_{\ep} ( u^*,v^*)$, for all $\gamma \in \mathcal{M}_{+}(X\times Y)$. Moreover, exploiting the fact that $u^*=v^{(c,\tau_1,\ep)}$ and $v^*=u^{(c,\tau_2,\ep)}$ one can directly show that $F^{\tau_1,\tau_2}_{\ep}(\gamma^*) = D^{\tau_1,\tau_2}_{\ep} ( u^*,v^*)$.

Using the above inequalities, we conclude that
$$F^{\tau_1,\tau_2}_{\ep}(\gamma) \geq  D^{\tau_1,\tau_2}_{\ep}( u^*,v^*) = F^{\tau_1,\tau_2}_{\ep}(\gamma^*) \geq  D^{\tau_1,\tau_2}_{\ep} ( u,v).$$
Notice that the inequality $F^{\tau_1,\tau_2}_{\ep}(\gamma) \geq F^{\tau_1,\tau_2}_{\ep}(\gamma^*)$ grant us that $\gamma^*$ is a minimizer for the primal problem \eqref{eq:UOTclassical} and that, in particular, we have $\cO\cT^{\tau_1,\tau_2}_{\ep}(\mu,\nu) = D^{\tau_1,\tau_2}_{\ep}(u^*,v^*)$.

Finally, we assume that $\cO\cT^{\tau_1,\tau_2}_{\ep}(\mu,\nu) = D^{\tau_1,\tau_2}_{\ep}(u^*,v^*)$ holds for some $u^*$ and $v^*$. We want to show that $u^*$ and $v^*$ are maximizers of the dual problem \eqref{eq:dualunbalanced}. Taking the minimum over $\gamma\in\mathcal{M}_{+}( X \times Y )$ in \eqref{eq:primaldual}, for any $u\in L^{\infty}(X)$ and $v\in L^\infty(Y)$ we have that
\[
\cO\cT^{\tau_1,\tau_2}_{\ep}(\mu,\nu) \geq D^{\tau_1,\tau_2}_{\ep}(u,v).
\]
By hypothesis $\cO\cT^{\tau_1,\tau_2}_{\ep}(\mu,\nu) = D^{\tau_1,\tau_2}_{\ep}(u^*,v^*)$. Then, we conclude that $D^{\tau_1,\tau_2}_{\ep}(u^*,v^*)\geq D^{\tau_1,\tau_2}_{\ep}(u,v),$ for all $u,v\in L^\infty$. So, $u^*$ and $v^*$ are maximizers of the dual problem.
\end{proof}

\section{Unbalanced Non-commutative Optimal Transport}\label{sec:unbalancedQOT}

Let $\cH_1$ and $\cH_2$ be Hilbert spaces with dimensions, respectively, $d_1$ and $d_2$, $C\in \rmH(\cH_1\otimes\cH_2)$ be a Hermitian operator on $\cH_1\otimes\cH_2$, $\rho\in\SDP(\cH_1)$ and $\sigma\in\SDP(\cH_2)$ be Hermitian and semi-definite positive operators on $\cH_1$ and $\cH_2$ respectively. 

Given $\ep,~\tau_1,~\tau_2>0$, define the functional $\Fept:\SDP(\cH_1\otimes \cH_2)
\to \R$ as
\begin{equation}
    \Fept(\Gamma) = \Tr[C\Gamma] + \ep\cS[\Gamma] + \tau_1\cE[\Gamma_1|\rho] + \tau_2\cE[\Gamma_2|\sigma],
\end{equation}
where $\cS[\Gamma]$ is the von Neumann Entropy 
\[
\cS:\SDP(\cH_1\otimes\cH_2)\to \R, \quad \quad \cS[\Gamma] = \Tr[\Gamma(\log\Gamma-\Id)],
\]
the matrices $\Gamma_1,\Gamma_2$ are respectively, the partial traces of $\Gamma_1 = \Tr_2[\Gamma]$, $\Gamma_2 = \Tr_1[\Gamma]$ in $\cH_2$ and $\cH_1$ and the functional $\cE[\cdot|\cdot]$ is the quantum (Umegaki) relative entropy 
\[
\cE:\SDP(\cH)\times \SDP(\cH)\to \R, \quad \cE[\rho_1|\rho_2] = 
	\begin{cases}
		\Tr[\rho_1(\log\rho_1 - \log\rho_2 - \Id)+\rho_2] &\text{if } \ker \rho_1 \subset \ker \rho_2, \\
		+\infty &\text{otherwise}
	\end{cases}
\]
between $\Gamma_1$ (resp. $\Gamma_2)$ and $\rho$ (resp. $\sigma$).

The von Neumann entropy regularized Unbalanced Non-commutative Optimal Transport is given by 
\begin{equation}\label{eq:UQOT}
\UQOT[\rho,\sigma] = \inf \left\lbrace \Fept(\Gamma) \suchthat \Gamma \in \SDP(\cH_1\otimes \cH_2) \right\rbrace.
\end{equation}





\subsection{Weak duality}

Let $\cH_1,~\cH_2$ be finite-dimensional Hilbert spaces with dimensions $d_1$ and $d_2$, respectively. Consider the dual functional $\Dept:\rmH(\cH_1)\times \rmH(\cH_2)\to \R$ defined by
\begin{equation}\label{eq:dualfunctional}
\Dept(U,V) = -\tau_1\Tr[(\exp\left(-\frac{U}{\tau_1}\right)-\Id)\rho]-\tau_2\Tr[(\exp\left(-\frac{V}{\tau_2}\right)-\Id)\sigma] - \varepsilon\Tr[\exp\left(\frac{U\oplus V-C}{\ep}\right)],
\end{equation}
where $U\oplus V = U\otimes \Id + \Id\otimes V$ denotes the Kronecker sum of $U$ and $V$. 

In the following, similarly to the Proposition \ref{prop:primaldual} for the classical case, we will derive the non-commutative weak duality formulation of~\eqref{eq:UQOT} using the Fenchel-Young inequality, which we will restrict to Hermitian elements $\rmH(\cH)\subset\rmB(\cH)$, instead of considering compact and trace class operators, which in the finite-dimensional case are the same as $\rmB(\cH)$.

\begin{deff}[Convex conjugate]
   Let $\cH$ be a $d$-dimensional Hilbert space and let $\rmH(\cH)\subset\rmB(\cH)$ be the space of Hermitian operators on $\cH$. Let $\Xi : \rmH(\cH) \to \R \cup \{+\infty\}$ be a proper function. Then the convex conjugate of $\Xi$ is a function $\Xi^* : \rmH(\cH) \to \R \cup \{+\infty\}$ defined as
   \[
       \Xi^*(B) = \sup_{A \in \rmH(\cH)} 
       \{ 
           \langle B, A \rangle - \Xi(A)
       \} = 
       \sup_{A \in \rmH(\cH)} 
       \{ 
            \Tr[AB] - \Xi(A)
       \},
   \]
\end{deff}

Notice that the inequality $\Xi(A)+\Xi^*(B)\geq \Tr[AB]$ holds for any $A,~B\in \rmH(\cH)$. 

Before we proceed to prove the weak duality, we will need the following proposition.

\begin{proposition}\label{prop:trace_convex_prod}
Let $\cH$ be a $d$-dimensional Hilbert space, let $A\in\rmH(\cH)$ be a Hermitian operator, and let $\rho\in\SDP(\cH)$ be a semi-definite positive Hermitian operator. Let $f:\R\to\R_+$ be a nonnegative convex function. Then the map $A\in\rmH(\cH)\mapsto \Tr[f(A)\rho]\in\R$ is convex, and it is strictly convex when $f$ is, and $\rho$ is definite positive.
\end{proposition}
\begin{proof}
Let $t\in(0,1)$ and let $A,B\in\rmH(\cH)$ be Hermitian. Consider the orthonormal basis $\{\ket{\phi_i}\}_{i=1}^d$ that diagonalizes $tA+(1-t)B$ 
and write $f(tA+(1-t)B)$ as 
\[
f(tA+(1-t)B) = \sum_{i=1}^d f(\bra{\phi_i}tA+(1-t)B\ket{\phi_i})\ket{\phi_i}\bra{\phi_i}.
\]
We can compute the $\Tr[f(tA+(1-t)B)\rho] = \Tr[\rho\, f(tA+(1-t)B)]$ directly as
\begin{align*}
    \Tr[f(tA+(1-t)B)\rho]  =&\sum_{i=1}^d \bra{\phi_i}\rho \,f(tA+(1-t)B)\ket{\phi_i}=\\
    =& \sum_{i=1}^d f(\bra{\phi_i}tA+(1-t)B\ket{\phi_i}) \bra{\phi_i}\rho\ket{\phi_i} ,
\end{align*}
which due to the convexity of $f$ can be bounded from above by
\begin{align*}
    \leq t\sum_{i=1}^d f(\bra{\phi_i}A\ket{\phi_i}) \bra{\phi_i}\rho\ket{\phi_i} +(1-t)  \sum_{i=1}^d f(\bra{\phi_i}B\ket{\phi_i}) \bra{\phi_i}\rho\ket{\phi_i},
\end{align*}
and the inequality is strict if $f$ is strictly convex and $\rho$ is definite positive.

Now, see that due to Operator Jensen's Inequality (e.g., Theorem 1.2 in \cite{PecFurMicSeo05}) we can use that $f(\bra{\phi_i}\cdot \ket{\phi_i})\leq \bra{\phi_i}f(\cdot)\ket{\phi_i}$, and thus
\begin{align*}
    \Tr[f(tA+(1-t)B)\rho] \leq & t\sum_{i=1}^d \bra{\phi_i}f(A)\ket{\phi_i} \bra{\phi_i}\rho\ket{\phi_i} +(1-t)  \sum_{i=1}^d \bra{\phi_i}f(B)\ket{\phi_i}\bra{\phi_i}\rho\ket{\phi_i}.
\end{align*}

On the other hand, due to the fact that $f\geq 0$ and $\rho\geq 0$, the operator $f(\cdot)\ket{\phi_k}\bra{\phi_k}\rho$ is semi-definite positive for any $k=1,\dots, d$, and in particular, since $\sum_{k=1}^d \ket{\phi_k}\bra{\phi_k}=\Id$, we have
\[\bra{\phi_i}f(\cdot)\ket{\phi_i} \bra{\phi_i}\rho\ket{\phi_i} \leq \bra{\phi_i}f(\cdot)\left(\sum_{k=1}^d \ket{\phi_k} \bra{\phi_k}\right)\rho\ket{\phi_i} = \bra{\phi_i}f(\cdot)\rho\ket{\phi_i},\]
which finally yields
\begin{align*}
    \Tr[f(tA+(1-t)B)\rho] \leq & t\sum_{i=1}^d \bra{\phi_i}f(A)\rho\ket{\phi_i} +(1-t)  \sum_{i=1}^d \bra{\phi_i}f(B)\rho\ket{\phi_i} = \\
    =& t\Tr[f(A)\rho] +(1-t)\Tr[f(B)\rho],
\end{align*}
where, as mentioned before, the inequality is strict for $f$ strictly convex and $\rho\in\DP(\cH)$. 
\end{proof}

\begin{teo}\label{thm:UQOT_weakdual}
Let $\varepsilon,\tau_1,\tau_2>0$ be positive numbers, $\cH_1$ and $\cH_2$ be finite-dimensional Hilbert spaces with dimensions, respectively, $d_1$ and $d_2$. Let $C\in \rmH(\cH_1\otimes\cH_2)$ be a Hermitian operator on $\cH_1\otimes\cH_2$, $\rho\in\DP(\cH_1),~\sigma\in\DP(\cH_2)$ be Hermitian definite positive operators on $\cH_1$ and $\cH_2$. Then,
\begin{equation}\label{eq:dualpbunb}
\UQOT[\rho,\sigma] = \inf_{\Gamma\in\SDP(\cH_1\otimes\cH_2)} \lbrace\Fept(\Gamma)\rbrace \geq \sup_{\substack{U \in \rmH(\cH_1)\\ V\in \rmH(\cH_2)}}\lbrace \Dept(U,V)  \rbrace, 
\end{equation}
where $\Dept:\rmH(\cH_1)\times \rmH(\cH_2)\to \R$ is the Kantorovich dual functional defined in \eqref{eq:dualfunctional}.
\end{teo}

\begin{proof}
Define the functions $\Xi:\rmH(\cH_1)\times \rmH(\cH_2)\to\R$ and $\Theta:\rmH(\cH_1\otimes\cH_2)\to\R$ by
\begin{equation}\label{Xi_U_V}
    \Xi(U,V) = \tau_1\Tr[(\exp\left(-\frac{U}{\tau_1}\right)-\Id)\rho]+\tau_2\Tr[(\exp\left(-\frac{V}{\tau_2}\right)-\Id)\sigma],
\end{equation}
\begin{equation}\label{Theta_Z}
    \Theta(Z) =
        \varepsilon \Tr[\exp\left(\frac{-Z-C}{\varepsilon}\right)],
\end{equation}
and define the operator A by
\begin{equation}\label{oper:A}
    A: 
    \begin{matrix}
    \rmH(\cH_1)\times \rmH(\cH_2) & \to &   \rmH(\cH_1\otimes\cH_2) \\
        (U, V)  &\mapsto  &-(U\oplus V).
    \end{matrix}
\end{equation}

Using the latter notation, the dual functional $\Dept$ defined in \eqref{eq:dualfunctional} can be expressed in terms of $\Xi$ and $\Theta$  as 
\begin{equation*}
    \Dept(U, V)=- \Xi(U,V)-\Theta(A(U,V)).
\end{equation*}

In the upcoming part of the proof, we will show that for any $\Gamma\in\rmH(\cH_1\otimes\cH_2)$ we obtain the following equality 
\begin{equation}\label{Fept:sum}
    \Theta^*(-\Gamma) + \Xi^*(A^*\Gamma) =\begin{cases}
        \Fept(\Gamma), &\Gamma\in\SDP(\cH_1\otimes\cH_2),\\
        +\infty, & \text{otherwise,}
    \end{cases} 
\end{equation}
where $A^*:\rmH(\cH_1\otimes\cH_2)\to\rmH(\cH_1)\times\rmH(\cH_2)$ is the adjoint operator for $A$, obtained via duality bracket
\begin{equation}\label{oper:Astar}
    \langle A^*\Gamma, (U,V)\rangle = -\langle \Gamma, U\oplus V\rangle = -\Tr[U\otimes \Id\Gamma ] - \Tr [\Id\otimes V\Gamma ] = -\Tr[ U\Gamma_1] -\Tr[ V\Gamma_2].
\end{equation}


First, consider the convex conjugate of $\Theta$ for any $\Gamma\in \rmH(\cH_1\otimes\cH_2)$
\begin{equation*}
    \Theta^*(-\Gamma) = \sup_{Z\in \rmH(\cH_1\otimes\cH_2)} \left\lbrace \Tr[-Z\Gamma -\ep e^{\frac{-Z-C}{\varepsilon}}]\right\rbrace.
\end{equation*}

Notice that if $\Gamma$ is not semi-definite positive, then there exists $\lambda_k\in \Sp(\Gamma)$ that is negative. Consider the diagonal decomposition $\Gamma = \sum_{i}\lambda_i|\psi_i\rangle\langle\psi_i|$ and take $Z_n = n |\psi_k\rangle\langle\psi_k|$, then
\begin{equation*}
    \Tr[-Z_n\Gamma] = -n\lambda_k \xrightarrow[n\to\infty]{} +\infty,
\end{equation*}
on the other hand, using the (strict) convexity of the trace function $A\mapsto \Tr[\exp(A)]$ (see e.g., Theorem 2.10 in \cite{carlen2010trace}), we get
\begin{equation*}
-\ep\Tr[e^{\frac{-Z_n-C}{\ep}}] \geq -\frac{\ep}{2}\Tr[e^{-\frac{2Z_n}{\ep}}] - \frac{\ep}{2}\Tr[e^{-\frac{2C}{\ep}}]= -\frac{\ep}{2} e^{-\frac{2n}{\ep}} -\frac{\ep}{2}\Tr[e^{-\frac{2C}{\ep}}] \xrightarrow[n\to\infty]{} -\frac{\ep}{2}\Tr[e^{-\frac{2C}{\ep}}],
\end{equation*}
which is a finite value, and whence $\Theta^*(-\Gamma)=+\infty$.
For $\Gamma$ semi-definite positive, invoking the variational principle, and using strict convexity of $\Theta$, we can obtain the unique maximizer $\Tilde{Z} = -C-\ep\log\Gamma$, which yields
\begin{equation}\label{Theta:ctransform}
\begin{aligned}
    \Theta^*(-\Gamma) =& \sup_{Z\in \rmH(\cH_1\otimes\cH_2)} \left\lbrace \Tr[-Z\Gamma -\ep\exp(\frac{-Z-C}{\varepsilon})]\right\rbrace =   \Tr[-\Tilde{Z}\Gamma -\ep\exp(\frac{-\Tilde{Z}-C}{\varepsilon})] =\\
    =&\Tr[-(-C-\ep\log\Gamma)\Gamma -\ep\Gamma] = \Tr[C\Gamma] +\varepsilon\Tr[\Gamma(\log\Gamma-\Id)] = \Tr[C\Gamma] + \varepsilon \cS[\Gamma],
\end{aligned}
\end{equation}
where we use the fact the trace of a product of two operators commutes.


Using a similar approach, we will compute $\Xi^*(A^*\Gamma)$, and in this case, the objective can be split into two functions of $U$ and $V$, which can be optimized independently

\begin{equation*}
\begin{aligned}
    \Xi^*(A^*\Gamma) = &\sup_{\substack{U\in\rmH(\cH_1)\\
    V\in\rmH(\cH_2)}} \left\lbrace 
    \langle A^*\Gamma, (U,V)\rangle -\tau_1\Tr[(\exp\left(-\frac{U}{\tau_1}\right)-\Id)\rho]-\tau_2\Tr[(\exp\left(-\frac{V}{\tau_2}\right)-\Id)\sigma] \right\rbrace = \\
    =&\sup_{U\in \rmH(\cH_1)} \left\lbrace 
    -\Tr [ U\Gamma_1] -\tau_1\Tr[\exp\left(-\frac{U}{\tau_1}\right)\rho]+\tau_1\Tr[\rho]\right\rbrace + 
    \\
    +&\sup_{V\in \rmH(\cH_2)} \left\lbrace-\Tr[ V\Gamma_2] - \tau_2\Tr[\exp\left(-\frac{V}{\tau_2}\right)\sigma] +\tau_2\Tr[\sigma] \right\rbrace,
\end{aligned}
\end{equation*}
For simplicity, we will consider only the component for $U$, as the result for $V$ is obtained in a similar manner.

By a similar argument that was used for computing $\Theta^*(-\Gamma)$ for $\Gamma$ not semi-definite positive, one can conclude that $\Xi^*(A^*\Gamma)=+\infty$, therefore it is sufficient to consider $\Gamma\in\SDP(\cH_1\otimes\cH_2)$. Taking the variation once more, we obtain a unique maximizer $\Tilde{U} = \tau_1\log\rho -\tau_1\log\Gamma_1$, since the map $U\mapsto \Tr[\exp(-\frac{U}{\tau_1})\rho]$ is strictly convex due to Proposition \ref{prop:trace_convex_prod}, and thus
\begin{equation*}
\begin{aligned}
    &\sup_{U\in \rmH(\cH_1)} \left\lbrace 
    -\Tr[ U\Gamma_1] -\tau_1\Tr[\exp(-\frac{U}{\tau_1})\rho]+\tau_1\Tr[\rho]\right\rbrace  = \Tr[-\Tilde{U}\Gamma_1-\tau_1 (\exp(-\frac{\Tilde{U}}{\tau_1})-\Id)\rho] = \\
    =& \Tr[-(\tau_1\log\rho -\tau_1\log \Gamma_1 )\Gamma_1 -\tau_1 (\Gamma_1-\rho)] = \tau_1\Tr[\Gamma_1(\log\Gamma_1-\log\rho-\Id)+\rho] = \tau_1\cE[\Gamma_1|\rho].
\end{aligned}
\end{equation*}
 
Consequently, the expression for V takes the following form
$$\sup\limits_{V\in \rmH(\cH_2)} \left\lbrace-\Tr[ V\Gamma_2] - \tau_2\Tr[(\exp(-\frac{V}{\tau_2})-\Id)\sigma] \right\rbrace = \tau_2 \cE[\Gamma_2|\sigma],$$
and thus for $\Gamma$ semi-definite positive
\begin{equation}\label{Xi:ctransform}
    \Xi^*(A^*\Gamma) = \tau_1\cE[\Gamma_1|\rho] +\tau_2 \cE[\Gamma_2|\sigma],
\end{equation}
and combining \eqref{Theta:ctransform} with \eqref{Xi:ctransform} yields $\Fept(\Gamma) = \Theta^*(-\Gamma) + \Xi^*(A^*\Gamma)$.


Finally, see that for $\Gamma\in\SDP(\cH_1\otimes\cH_2)$, $U\in\rmH(\cH_1)$, $V\in\rmH(\cH_2)$ and choosing $Z = A(U,V) = -U\oplus V$, we get
\begin{align*}
    \Fept(\Gamma) \geq  \Tr[(-U\oplus V)(-\Gamma)] -\Theta(-U\oplus V) + \Tr[(-U\oplus V)\Gamma] -\Xi(U,V) = \Dept(U,V),
\end{align*}
and taking respective infimum and supremum, we conclude the proof.
\end{proof}

\begin{rem}
    One can notice that in the proof above we show the one-sided inequality related to the Fenchel-Rockafellar Theorem, restricted to the real subspace of bounded operators in finite dimensions 
    \[
\inf_{(U,V)\in\rmH(\cH_1)\times\rmH(\cH_2)} 
        \{
            \Theta(A(U,V)) + \Xi(U,V)
        \} \geq 
        \sup_{\Gamma \in \rmH(\cH_1\otimes \cH_2)} 
        \{
            -\Theta^*(-\Gamma) - \Xi^*(A^*\Gamma)
        \}.
    \]
\end{rem}




\begin{rem}\label{Fept:lsc}
Notice that by the properties of the Legendre-Fenchel transform, it follows that both $\Theta^*$ and $\Xi^*$ are convex and lower semicontinuous and, as a consequence of \eqref{Theta:ctransform} and \eqref{Xi:ctransform}, so are the von Neumann entropy, the relative entropy, and the primal functional $\Fept$. 


\end{rem}

\subsection{Existence of a minimizer}

In this section, we will show the existence of a minimizer of the primal problem (\ref{eq:UQOT}). An important part of the proof will play the following lemma, which gives the necessary compactness properties. 

\begin{lemma}[$\Fept(\cdot)$ is coercive]\label{Gamma:compactness}
    Let $\ep, \tau_1,\tau_2>0$ be positive parameters, $\cH_1$ and $\cH_2$ be finite-dimensional Hilbert spaces with dimensions, respectively, $d_1$ and $d_2$. Let $C\in \rmH(\cH_1\otimes\cH_2)$ be a Hermitian operator, $\rho\in\SDP(\cH_1),~\sigma\in\SDP(\cH_2)$ be Hermitian semi-definite positive operators. Assume that $\lbrace\Gamma^n\rbrace_{n\geq 1}$ is a sequence in $\SDP(\cH_1\otimes\cH_2)$ such that   
    \[
    \sup\limits_{n\geq 1} \lbrace\Fept(\Gamma^n)\rbrace < \infty.
    \]
Then there exists $\Gamma^0\in\SDP(\cH_1\otimes\cH_2)$ and a subsequence $\lbrace\Gamma^{n_k}\rbrace_{k\geq 1}$ such that $\Gamma^{n_k}$ converges to $\Gamma^0$ in the weak*-topology.
\end{lemma}
\begin{proof} 
Let $\lbrace\Gamma^n\rbrace_{n\geq 1}$ be a sequence in $\SDP(\cH_1\otimes\cH_2)$ such that   $\sup\limits_{n\geq 1} \lbrace\Fept(\Gamma^n)\rbrace < \infty$. It suffices to show that $\lbrace||\Gamma^n||_\infty\rbrace_{n\geq 1}$ is bounded and apply the Banach-Alaoglu theorem. 

Suppose that $\lbrace\Gamma^n\rbrace_{n\geq 1}$ is not bounded. Then one can extract a subsequence $\lbrace\Gamma^{n_k}\rbrace_{k\geq 1}$ such that $||\Gamma^{n_k}||_\infty\to\infty$. For every $n_k$, let us consider the diagonal decompositions $\Gamma^{n_k} = \sum_{i=1}^{d_1 d_2}\lambda^{n_k}_i|\psi^{n_k}_i\rangle\langle\psi^{n_k}_i|$. Then


\begin{align*}
    \Tr[C\Gamma^{n_k}]+\ep\cS[\Gamma] =& \sum_{i=1}^{d_1 d_2} \bra{\psi^{n_k}_i}C\Gamma \ket{\psi^{n_k}_i} + \ep\sum_{i=1}^{d_1 d_2}\lambda^{n_k}_i(\log\lambda^{n_k}_i -1)  =\\
    =& \sum_{i=1}^{d_1 d_2} \left(\bra{\psi^{n_k}_i}C \ket{\psi^{n_k}_i}\lambda^{n_k}_i +\ep \lambda^{n_k}_i(\log\lambda^{n_k}_i -1) \right) \geq \\
    \geq & \sum_{i=1}^{d_1 d_2} \left(-||C||_{\infty}\lambda^{n_k}_i +\ep \lambda^{n_k}_i(\log\lambda^{n_k}_i -1) \right) \to+\infty,
\end{align*}
as one of the eigenvalues $\lambda^{n_k}_i\to +\infty$ and the function $t\mapsto \ep t(\log t -1)$ is superlinear at infinity. In addition, due to Klein's inequality $\tau_1\cE[\Gamma^{n_k}_1|\rho]+\tau_2\cE[\Gamma^{n_k}_2|\sigma]\geq 0$, we finally get that $\Fept(\Gamma^{n_k})\to\infty$, which is a contradiction with the fact that $\sup\limits_{n\geq 1} \lbrace\Fept(\Gamma^n)\rbrace < \infty$. Thus, the sequence $\lbrace||\Gamma^n||_\infty\rbrace_{n\geq 1}$ is bounded and, by Banach-Alaoglu theorem, admits a weakly*-converging subsequence.
\end{proof}

\begin{proposition}[Existence of a minimizer in \eqref{eq:UQOT}]\label{Fept:minimizer}
Let $\ep, \tau_1,\tau_2>0$ be positive parameters, $\cH_1$ and $\cH_2$ be finite-dimensional Hilbert spaces, $C\in \rmH(\cH_1\otimes\cH_2)$ be a Hermitian operator on $\cH_1\otimes\cH_2$ and, $\rho\in\SDP(\cH_1),~\sigma\in\SDP(\cH_2)$ be Hermitian semi-definite positive operators on $\cH_1$ and $\cH_2$. Then there exists a minimizer for the von Neumann entropy regularized Unbalanced Non-commutative Optimal Transport in \eqref{eq:UQOT}.
\end{proposition}

\begin{proof}
The proof follows from the direct method of Calculus of Variations. Set $m=\inf\limits_{\Gamma\in\SDP(\cH_1\otimes\cH_2)} \Fept(\Gamma)$ and consider $\lbrace\Gamma^n\rbrace_{n\geq 1}$ a minimizing sequence such that 
\[
\Fept(\Gamma^n)\leq m+\frac{1}{n}\leq m+1.
\]
By Lemma \ref{Gamma:compactness}, there exist $\Gamma^0\in\SDP(\cH_1\otimes\cH_2)$ and a subsequence $\lbrace\Gamma^{n_k}\rbrace_{k\geq 1}$ such that $\Gamma^{n_k}$ weakly*-converges to $\Gamma^0$. Finally, using lower semicontinuity of $\Fept$ (see Remark \ref{Fept:lsc}), we obtain the required chain of inequalities
\begin{equation*}
    m\leq \Fept(\Gamma^0) \leq \liminf\limits_{k\to\infty}\Fept(\Gamma^{n_k}) \leq \liminf\limits_{k\to\infty} m+\frac{1}{n_k} = m,
\end{equation*}
whence the problem \eqref{eq:UQOT} admits a minimizer $\Gamma^0$.
\end{proof}

\subsection{Gamma-convergence when $\ep\to 0^+$ and when $\tau_1=\tau_2\to +\infty$}\label{sec:unbalancedgammaQOT}

We focus on computing the limit cases when $\ep\to 0^+$ or when $\tau_1=\tau_2\to +\infty$ of the primal unbalanced functional $\Fept(\Gamma)$.
Consider the families 
\begin{equation}\label{Fep:family}
    \lbrace\Fept\rbrace_{\ep>0} = \lbrace \Fept(\cdot) = \Tr[C(\cdot)] + \ep\cS[\cdot] + \tau_1\cE[(\cdot)_1|\rho] + \tau_2\cE[(\cdot)_2|\sigma] ,~\ep>0\rbrace,
\end{equation}
and 
\begin{equation}\label{Ftau:family}    \lbrace\cF^{\tau}_{\ep}\rbrace_{\tau>0} = \lbrace \Tr[C(\cdot)] + \ep\cS[\cdot] + \tau\left(\cE[(\cdot)_1|\rho] + \cE[(\cdot)_2|\sigma]\right),~\tau>0\rbrace.
\end{equation}

First, we need to use the following notions.
\begin{deff}[$\Gamma$-convergence, \cite{BraidesA02}]
Let $(\cX, \cT)$ be a topological space. Given a family of functionals
$\cF_{\alpha}: \cX \to \R$ we say that it $\Gamma$-converges to $\cF: \cX \to \R$ in the topology $\cT$ if it satisfies
\begin{enumerate}
    \item[(i)] $\liminf$ inequality
    
    For every $x_{\alpha}\to x$ in $\cT$ holds
\begin{equation}\label{Gamma:liminf}
        \cF(x)\leq \liminf_{\alpha} \cF_{\alpha}(x_{\alpha})
    \end{equation}

    \item[(ii)] $\limsup$ inequality

    For every $x\in\cX$ there exists a sequence $x_{\alpha}\to x$ in $\cT$ such that
\begin{equation}\label{Gamma:limsup}
        \cF(x)\geq \limsup_{\alpha} \cF_{\alpha}(x_{\alpha}).
    \end{equation}   
\end{enumerate}
Or equivalently, if $\lbrace\cF_{\alpha}\rbrace_{\alpha}$ satisfies \eqref{Gamma:liminf}, the condition \eqref{Gamma:limsup} is equivalent to 
\begin{enumerate}
    \item[(ii)'] recovery sequence

    For every $x\in\cX$ there exists a sequence $x_{\alpha}\to x$ in $\cT$ such that
\begin{equation}\label{Gamma:recovery}
        \cF(x)=\lim_{\alpha} \cF_{\alpha}(x_{\alpha}).
    \end{equation}   
\end{enumerate} 
\end{deff}

\begin{deff}[Equi-coercivity]
Let $(\cX, \cT)$ be a topological space. Given a family of functionals $\cF_{\alpha}: \cX \to \R$ we say that it is equi-coercive if for any sequence $\lbrace x_{\alpha}\rbrace_{\alpha}$ such that
\begin{equation}\label{F:equi-coercive}
    \sup_{\alpha}|\cF_{\alpha}(x_{\alpha})|<\infty
\end{equation}
there exists a converging subsequence in topology $\cT$.
\end{deff}

In the following we will compute the $\Gamma$-limits of $\lbrace\Fept\rbrace_{\ep}$ and $\lbrace\cF^{\tau}_{\ep}\rbrace_{\tau}$, and in particular, we will verify that the sequences of minimizers of regularized problems will converge to the minimizers of the $\Gamma$-limits via direct application of the following Theorem.
\begin{teo}[Theorem 7.8, \cite{DalMaso93}]\label{Gamma:fundamental}
Let $(\cX, \cT)$ be a topological space and let $\lbrace \cF_{\alpha}\rbrace_{\alpha}$ be a family of functionals $\cF_{\alpha}:\cX\to\R$. Suppose that $\cF_{\alpha}$ $\Gamma$-converges to $\cF:\cX\to\R$ in $(\cX, \cT)$ and $\lbrace\cF_{\alpha}\rbrace_{\alpha}$ is equi-coercive. Then $\cF$ is coercive and
\begin{equation*}
    \min_{x\in X}\cF(x) = \lim_{\alpha} \inf_{x\in X} \cF_{\alpha}(x).
\end{equation*}
\end{teo}

To begin with, notice that the direct consequence of Lemma \ref{Gamma:compactness} is the equi-coercivity property. 

\begin{lemma}[Equi-coercivity]\label{Fept:equicoercive}
Let $\cH_1$ and $\cH_2$ be finite-dimensional Hilbert spaces with dimensions $d_1$ and $d_2$, respectively. Let $C\in \rmH(\cH_1\otimes\cH_2)$ be a Hermitian operator on $\cH_1\otimes\cH_2$, $\rho\in\SDP(\cH_1),~\sigma\in\SDP(\cH_2)$ be Hermitian semi-definite positive operators on $\cH_1$ and $\cH_2$. Then the families of functionals $\lbrace\Fept\rbrace_{\ep>0}$ and $\lbrace\cF^{\tau}_{\ep}\rbrace_{\tau>0}$ are equi-coercive.
\end{lemma}

\begin{proof}
What we are interested in are the cases when $\ep\to 0^+$ and $\tau\to\infty$. 
Notice again that by Klein's inequality and the fact that $x\log x\geq x-1$, we obtain for $\lbrace\Fept(\Gamma^{\ep})\rbrace_{\ep>0}$
\begin{equation*}
    \ep\cS[\Gamma^{\ep}] +\tau_1\cE[\Gamma^{\ep}_1|\rho]+\tau_2\cE[\Gamma^{\ep}_2|\sigma]\geq -\ep\Tr[\Id]>-\infty \text{~~~~~~for bounded }\ep>0,
\end{equation*}
and for $\lbrace\cF^{\tau}_{\ep}(\Gamma^{\tau})\rbrace_{\tau>0}$
\begin{equation*}
    \ep\cS[\Gamma^{\ep}] +\tau(\cE[\Gamma^{\ep}_1|\rho]+\cE[\Gamma^{\ep}_2|\sigma])\geq -\ep\Tr[\Id]>-\infty \text{~~~~~~for any }\tau>0,
\end{equation*}
thus both lower bounds are independent of $\Gamma$.

Similarly to the proof of Lemma \ref{Gamma:compactness}, assume that the sequences $\lbrace\Gamma^{\ep}\rbrace_{\ep>0}$ and $\lbrace\Gamma^{\tau}\rbrace_{\tau>0}$ are not bounded and extract subsequences $\ep(k)$ and $\tau(h)$ such that $||\Gamma^{\ep(k)}||_\infty\to\infty$ and $||\Gamma^{\tau(h)}||_\infty\to\infty$. By the same superlinearity argument, for $\{\Gamma^{\tau(h)}\}_h$ we have
\begin{align*}
    \Tr[C\Gamma^{\tau(h)}] + \ep\cS[\Gamma^{\tau(h)}] + \tau\left(\cE[\Gamma^{\tau(h)}_1|\rho] + \cE[\Gamma^{\tau(h)}_2|\sigma]\right) \geq \Tr[C\Gamma^{\tau(h)}] + \ep\cS[\Gamma^{\tau(h)}] \to +\infty.
\end{align*}
For $\{\Gamma^{\ep(k)}\}_k$ we can notice that either $||\Gamma^{\ep(k)}_1||_\infty\to\infty$ and/or $||\Gamma^{\ep(k)}_2||_\infty\to\infty$, then write them in diagonal forms 
$$\Gamma^{\ep(k)}_1 = \sum_{i=1}^{d_1}\alpha^{\ep(k)}_i\ket{\phi^{\ep(k)}_i}\bra{\phi^{\ep(k)}_i}, \quad\quad \text{and}\quad\quad \Gamma^{\ep(k)}_2 = \sum_{j=1}^{d_2}\beta^{\ep(k)}_j\ket{\xi^{\ep(k)}_j}\bra{\xi^{\ep(k)}_j},$$
where $\{\ket{\phi^{\ep(k)}_i}\}_{i=1}^{d_1}$, $\{\ket{\xi^{\ep(k)}_j}\}_{j=1}^{d_2}$, $\{\ket{\phi^{\ep(k)}_i\otimes \xi^{\ep(k)}_j}\}_{i,j=1}^{d_1,d_2}$ are orthonormal basises of $\cH_1$, $\cH_2$, $\cH_1\otimes\cH_2$, respectively. We can analyze the relative entropy terms
\begin{align*}
    \cE[\Gamma^{\ep(k)}_1|\rho] = &\Tr[\Gamma^{\ep(k)}_1(\log(\Gamma^{\ep(k)}_1)-\log\rho-\Id)]+\Tr[\rho] = \\
    = & \sum_{i=1}^{d_1} \left(\alpha^{\ep(k)}_i\log\alpha^{\ep(k)}_i -\alpha^{\ep(k)}_i\bra{\phi^{\ep(k)}_i}\log\rho\ket{\phi^{\ep(k)}_i} - \alpha^{\ep(k)}_i \right)+\Tr[\rho],\\
    \cE[\Gamma^{\ep(k)}_2|\sigma] = &\Tr[\Gamma^{\ep(k)}_2(\log(\Gamma^{\ep(k)}_2)-\log\sigma-\Id)]+\Tr[\sigma] = \\
    = & \sum_{j=1}^{d_2} \left(\beta^{\ep(k)}_j\log\beta^{\ep(k)}_j -\beta^{\ep(k)}_j\bra{\xi^{\ep(k)}_j}\log\sigma\ket{\xi^{\ep(k)}_j} - \beta^{\ep(k)}_j \right)+\Tr[\sigma],
\end{align*}
and putting them together in $\Fept(\Gamma^{\ep(k)})$ yields
\begin{align*}
    \Fept(\Gamma^{\ep(k)})=&\Tr[C\Gamma^{\ep(k)}] + \ep\cS[\Gamma^{\ep(k)}] + \tau_1\cE[\Gamma^{\ep(k)}_1|\rho] + \tau_2\cE[\Gamma^{\ep(k)}_2|\sigma] \geq \\
    \geq &\Tr[C\Gamma^{\ep(k)}] - \ep\Tr[\Id] + \tau_1\cE[\Gamma^{\ep(k)}_1|\rho] + \tau_2\cE[\Gamma^{\ep(k)}_2|\sigma] = \\
    = & -\ep\Tr[\Id]+\tau_1\Tr[\rho]+\tau_2\Tr[\sigma] +\sum_{i,j=1}^{d_1,d_2} \bra{\phi^{\ep(k)}_i\otimes \xi^{\ep(k)}_j} C\Gamma \ket{\phi^{\ep(k)}_i\otimes \xi^{\ep(k)}_j} +\\
    & +\sum_{i=1}^{d_1} \alpha^{\ep(k)}_i\left(\log\alpha^{\ep(k)}_i -\bra{\phi^{\ep(k)}_i}\log\rho\ket{\phi^{\ep(k)}_i} - 1 \right) +\\
    &+\sum_{j=1}^{d_2} \beta^{\ep(k)}_j\left(\log\beta^{\ep(k)}_j -\bra{\xi^{\ep(k)}_j}\log\sigma\ket{\xi^{\ep(k)}_j} - 1 \right) \to+\infty,
\end{align*}
as the terms $x\mapsto x(\log x - a)$ are superlinear at infinity.

Whence, we get $\Fept(\Gamma^{\ep(k)})\to\infty$ and $\cF^{\tau}_{\ep}(\Gamma^{\tau(h)})\to\infty$, which contradicts the hypothesis that 
\begin{equation*}
\sup\limits_{\ep>0}\lbrace\Fept(\Gamma^{\ep})\rbrace<\infty  \text{~~~~~ and ~~~~~} \sup\limits_{\tau>0}\lbrace\cF^{\tau}_{\ep}(\Gamma^{\tau})\rbrace<\infty.
\end{equation*}

Consequently, the sequences are bounded, and due to the Banach-Alaoglu theorem, we can extract convergent subsequences.
\end{proof}

\begin{proposition}[$\Gamma-$convergence when $\ep\to 0^+$]\label{Gamma:ep0}
Let $\tau_1,\tau_2>0$ be positive parameters, $\cH_1$ and $\cH_2$ be finite-dimensional Hilbert spaces. Let $C\in \rmH(\cH_1\otimes\cH_2)$ be a Hermitian operator on $\cH_1\otimes\cH_2$, $\rho\in\SDP(\cH_1),~\sigma\in\SDP(\cH_2)$ be Hermitian semi-definite positive operators on $\cH_1$ and $\cH_2$. Define the Unbalanced Non-commutative Optimal Transport functional
\begin{equation}\label{F0}
    \cF^{\tau_1,\tau_2}_0(\Gamma) = \Tr[C\Gamma] + \tau_1\cE[\Gamma_1|\rho] + \tau_2\cE[\Gamma_2|\sigma].
\end{equation}

Then the family $\lbrace\Fept\rbrace_{\ep>0}$ defined in \eqref{Fep:family} $\Gamma$-converges to $\cF^{\tau_1, \tau_2}_0$ in the weak*-topology as $\ep\to 0^+$.
\end{proposition}

\begin{proof}
We need to verify the conditions of $\Gamma$-convergence. First, take a sequence $\lbrace\Gamma^{\ep}\rbrace_{\ep}$ such that $\Gamma^{\ep}\rightharpoonup^*\Gamma^0$ when $\ep\to 0^+$. In particular, this also implies that the sequence $\lbrace\Gamma^{\ep}\rbrace_{\ep}$ is bounded, and whence $\sup\limits_{\ep> 0}|\cS[\Gamma^{\ep}]|\leq M<\infty$ for some $M\in\R$.

By Remark \ref{Fept:lsc}, the relative entropies are lower semicontinuous, thus
\begin{equation*}
    \begin{aligned}
        \cF^{\tau_1,\tau_2}_0(\Gamma^0) =& \liminf_{\ep\to 0} \left\lbrace \Tr[C\Gamma^{\ep}] +  \ep \cS[\Gamma^{\ep}]\right\rbrace+\tau_1\cE[\Gamma_1^0|\rho] + \tau_2\cE[\Gamma_2^0|\sigma] \leq\\
        \leq & \liminf_{\ep\to 0} \left\lbrace \Tr[C\Gamma^{\ep}] +  \ep \cS[\Gamma^{\ep}]+\tau_1\cE[\Gamma_1^{\ep}|\rho] + \tau_2\cE[\Gamma_2^{\ep}|\sigma] \right\rbrace = \liminf_{\ep\to 0} \Fept(\Gamma^{\ep}).
    \end{aligned}
\end{equation*}

Next, for arbitrary $\Gamma^0$ and we can take a constant recovery sequence, then 
\begin{equation*}
    \lim_{\ep\to 0 } \Fept(\Gamma^0) = \lim_{\ep\to 0}\cF^{\tau_1,\tau_2}_0(\Gamma^0)  + \ep \cS[\Gamma^0] = \cF^{\tau_1,\tau_2}_0(\Gamma^0) +\lim_{\ep\to 0}\ep \cS[\Gamma^0] = \cF^{\tau_1,\tau_2}_0(\Gamma^0),
\end{equation*}
therefore both conditions are satisfied, and whence the $\Gamma$-convergence holds.
\end{proof}

\begin{proposition}[$\Gamma-$convergence when $\tau=\tau_1=\tau_2\to+\infty$]\label{Gamma:tauinfty}
Let $\ep>0$ be positive parameter, $\cH_1$ and $\cH_2$ be finite-dimensional Hilbert spaces. Let $C\in \rmH(\cH_1\otimes\cH_2)$ be a Hermitian operator on $\cH_1\otimes\cH_2$, $\rho\in\SDP(\cH_1),~\sigma\in\SDP(\cH_2)$ be Hermitian semi-definite positive operators on $\cH_1$ and $\cH_2$. Then the family $\lbrace\cF^{\tau}_{\ep}\rbrace_{\tau}$  defined in \eqref{Ftau:family}
$\Gamma$-converges to $\cF^{\infty}_{\ep}$ in weak*-topology when $\tau\to\infty$, where
\begin{equation}\label{Finfty}
    \cF^{\infty}_{\ep}(\Gamma) = \Tr[C\Gamma] + \ep\cS[\Gamma] + \begin{cases}
        0, & \Gamma_1=\rho \text{ and } \Gamma_2=\sigma,\\
        +\infty, & \text{otherwise},
    \end{cases}
\end{equation}
which, in particular, is equal to the von Neumann entropy-regularized Non-commutative Optimal Transport functional from \eqref{QOT} when $\rho$ and $\sigma$ are density matrices.
\end{proposition}

\begin{proof}
Similarly, we will split the proof into two parts to verify the necessary conditions. 

First, take $\lbrace\Gamma^{\tau}\rbrace_{\tau}$ such that $\Gamma^{\tau}\rightharpoonup^*\Gamma^{\infty}$ when $\tau\to\infty$. By lower semicontinuity of $\Tr[C\Gamma]+\ep\cS[\Gamma]$ (see Remark \ref{Fept:lsc}) we directly get
\begin{equation}\label{TrCG+entr:lsc}
    \Tr[C\Gamma^{\infty}]+\ep\cS[\Gamma^{\infty}]\leq \liminf_{\tau\to\infty} \lbrace\Tr[C\Gamma^{\tau}]+\ep\cS[\Gamma^{\tau}]\rbrace.
\end{equation}

On the other hand, we need to consider the following two cases and use the lower semicontinuity of the relative entropies, 

\begin{itemize}
    \item If $\Gamma^{\infty}_1=\rho$ and $\Gamma^{\infty}_2=\sigma$, then $\cE[\Gamma^{\infty}_1|\rho]=  \cE[\Gamma^{\infty}_2|\sigma]=0$ and whence
    \begin{equation*}
        0 = \cE[\Gamma^{\infty}_1|\rho]+  \cE[\Gamma^{\infty}_2|\sigma]\leq \liminf_{\tau\to\infty}(\cE[\Gamma^{\tau}_1|\rho]+  \cE[\Gamma^{\tau}_2|\sigma]) \leq \liminf_{\tau\to\infty}\tau(\cE[\Gamma^{\tau}_1|\rho]+  \cE[\Gamma^{\tau}_2|\sigma]).
    \end{equation*}

    \item If $\Gamma^{\infty}_1\neq\rho$ or $\Gamma^{\infty}_2\neq\sigma$, then exists $\delta>0$ such that $\cE[\Gamma^{\infty}_1|\rho]+  \cE[\Gamma^{\infty}_2|\sigma]=
    \delta$, which along with lower semicontinuity gives
    \begin{equation*}
        \liminf_{\tau\to\infty} \tau(\cE[\Gamma^{\tau}_1|\rho]+  \cE[\Gamma^{\tau}_2|\sigma])\geq \lim_{\tau\to\infty} \tau\delta  = +\infty.
    \end{equation*}    
\end{itemize}
Combining both cases above with \eqref{TrCG+entr:lsc}, the $\liminf$ inequality follows.

Next, for arbitrary $\Gamma^{\infty}$, take the constant recovery sequence $\lbrace\Gamma^{\infty}\rbrace$. 
\begin{itemize}
    \item If $\Gamma^{\infty}_1=\rho$ and $\Gamma^{\infty}_2=\sigma$, then $\cE[\Gamma^{\tau}_1|\rho]=  \cE[\Gamma^{\tau}_2|\sigma]=0$.
    \item Conversely, if $\Gamma^{\infty}_1\neq\rho$ or $\Gamma^{\infty}_2\neq\sigma$, then again $\cE[\Gamma^{\infty}_1|\rho]+  \cE[\Gamma^{\infty}_2|\sigma]=
    \delta>0$
\end{itemize}
Summarizing both cases, we obtain that the condition \eqref{Gamma:recovery} is satisfied 
\begin{equation*}
\begin{aligned}
    \lim_{\tau\to\infty} \cF^{\tau}_{\ep}(\Gamma^{\infty}) &= \lim_{\tau\to\infty} \Tr[C\Gamma^{\infty}]+\ep\cS[\Gamma^{\infty}] +\begin{cases}
        0, & \Gamma^{\infty}_1=\rho \text{~and~} \Gamma^{\infty}_2=\sigma\\
        \tau\delta, & \text{otherwise}         
    \end{cases} = \\
    &=\Tr[C\Gamma^{\infty}]+\ep\cS[\Gamma^{\infty}] +\begin{cases}
        0, & \Gamma^{\infty}_1=\rho \text{~and~} \Gamma^{\infty}_2=\sigma\\
        +\infty, & \text{otherwise}  
    \end{cases}
    = \cF^{\infty}_{\ep}(\Gamma^{\infty}),   
\end{aligned}
\end{equation*}
which concludes the proof.
\end{proof}

The final result of this section establishes the convergence of the minimizers of the regularized problems $\lbrace\Fept\rbrace_{\ep>0} $ and $ \lbrace\cF^{\tau}_{\ep}\rbrace_{\tau>0}$, and as a consequence, the existence of the minimizers in the $\Gamma-$limits.
\begin{teo}[Convergence of the minima when $\ep\to 0^+$ or $\tau\to+\infty$]\label{thm:UQOT_convergence_minimizer}
Let $\cH_1$ and $\cH_2$ be finite-dimensional Hilbert spaces. Let $C\in \rmH(\cH_1\otimes\cH_2)$ be a Hermitian operator on $\cH_1\otimes\cH_2$, $\rho\in\SDP(\cH_1),~\sigma\in\SDP(\cH_2)$ be Hermitian semi-definite positive operators on $\cH_1$ and $\cH_2$. 
Consider the families $\lbrace\Fept\rbrace_{\ep>0}, \lbrace\cF^{\tau}_{\ep}\rbrace_{\tau>0}$ and the respective sequences of their minimizers $\lbrace\Gamma^{\ep}\rbrace_{\ep}, \lbrace\Gamma^{\tau}\rbrace_{\tau}$. Then 

\begin{itemize}
\item[(i)] There exists $\Gamma^0$ such that, up to taking subsequences, $\Gamma^{\ep}\rightharpoonup^*\Gamma^0$ and 
\begin{equation*}
    \lim_{\ep\to 0}\Fept(\Gamma^{\ep}) = \min_{\Gamma\in\SDP(\cH_1\otimes\cH_2)} \cF^{\tau_1,\tau_2}_0 (\Gamma) = \cF^{\tau_1,\tau_2}_0 (\Gamma^0). 
\end{equation*}

\item[(ii)] If $\Tr[\rho]=\Tr[\sigma]$, there exists $\Gamma^{\infty}$ such that, up to taking subsequences, $\Gamma^{\tau}\rightharpoonup^*\Gamma^{\infty}$ and
\begin{equation*}
    \lim_{\tau\to \infty}\cF^{\tau}_{\ep}(\Gamma^{\tau}) = \min \lbrace \cF^{\infty}_{\ep} (\Gamma)\suchthat \Gamma\in\SDP(\cH_1\otimes\cH_2),~ \Gamma_1=\rho,~\Gamma_2=\sigma \rbrace= \cF^{\infty}_{\ep} (\Gamma^{\infty}),
\end{equation*}
which coincides with the solution of the von Neumann entropy regularized Non-commutative Optimal transport problem, provided that $\rho$ and $\sigma$ are density matrices \eqref{QOT}.
\end{itemize}
\end{teo}
\begin{proof}
By Lemma \ref{Fept:equicoercive} and Propositions \ref{Gamma:ep0} and \ref{Gamma:tauinfty} we get the sufficient conditions to apply Theorem \ref{Gamma:fundamental}, hence
\begin{equation*}
    \min_{\Gamma\in \SDP(\cH_1\otimes\cH_2)}\cF^{\tau_1,\tau_2}_0 (\Gamma) = \lim_{\ep\to 0} \min_{\Gamma\in \SDP(\cH_1\otimes\cH_2)} \Fept(\Gamma)=\lim_{\ep\to 0}\Fept(\Gamma^{\ep}), 
\end{equation*}
and
\begin{equation*}
    \min_{\Gamma\mapsto (\rho,\sigma)}\cF^{\infty}_{\ep} (\Gamma) = \lim_{\tau\to \infty} \min_{\Gamma\in \SDP(\cH_1\otimes\cH_2)} \cF^{\tau}_{\ep}(\Gamma) = \lim_{\tau\to \infty} \cF^{\tau}_{\ep}(\Gamma^{\tau}).
\end{equation*}

In particular, this implies that the sequences of the minimum are bounded, namely  
\begin{equation*}
\sup\limits_{\ep>0}\lbrace\Fept(\Gamma^{\ep})\rbrace<\infty  \text{~~~~~ and ~~~~~} \sup\limits_{\tau>0}\lbrace\cF^{\tau}_{\ep}(\Gamma^{\tau})\rbrace<\infty.
\end{equation*}
By equi-coercivity, we can find $\Gamma^0, \Gamma^{\infty}\in\SDP(\cH_1\otimes\cH_2)$ and extract the subsequences $\ep(k), \tau(h)$ such that $\Gamma^{\ep(k)}\rightharpoonup^*\Gamma^0$ and $\Gamma^{\tau(h)}\rightharpoonup^*\Gamma^{\infty}.$
Consequently, 
\begin{equation*}
    \min_{\Gamma\in \SDP(\cH_1\otimes\cH_2)}\cF^{\tau_1,\tau_2}_0 (\Gamma) \leq \cF^{\tau_1,\tau_2}_0 (\Gamma^0)\leq \liminf_{k\to\infty} \cF^{\tau_1,\tau_2}_{\ep(k)} (\Gamma^{\ep(k)}) = \min_{\Gamma\in \SDP(\cH_1\otimes\cH_2)}\cF^{\tau_1,\tau_2}_0 (\Gamma),
\end{equation*}
and similarly
\begin{equation*}
    \min_{\Gamma\mapsto (\rho,\sigma)}\cF^{\infty}_{\ep} (\Gamma) \leq \cF^{\infty}_{\ep} (\Gamma^{\infty})\leq \liminf_{h\to\infty} \cF^{\tau(h)}_{\ep} (\Gamma^{\tau(h)}) = \min_{\Gamma\mapsto (\rho,\sigma)}\cF^{\infty}_{\ep} (\Gamma).
\end{equation*}
Therefore, it follows that $\cF^{\tau_1,\tau_2}_{0}$ (resp. $\cF^{\infty}_{\ep}$) admits a minimizer $\Gamma^0$ (resp. $\Gamma^{\infty}$).
\end{proof}

\section*{Acknowledgment}
The authors are very grateful to Emanuele Caputo for the feedback and useful suggestions for the first draft of the article. Both AG and NM acknowledge support of their research by the Canada Research Chairs Program, the Natural Sciences and Engineering Research Council of Canada and the New Frontiers in Research Fund [NFRFE-2021-00798].

\bibliographystyle{siam}
\bibliography{refs}

\end{document}